\documentclass[english, papersaver2]{llncs}
\usepackage{amsmath}
\usepackage{amssymb}
\usepackage{tikz}
\usepackage{mathrsfs}
\newcounter{contatore}

\newcommand{\osc}{\mathrm{osc}}
\newcommand{\sdiff}{\bigtriangleup}
\newcommand{\spazio}{\textrm{ }}
\newcommand{\restr}{\upharpoonright}

\newcommand{\pow}[1]{\mathscr{P}(#1)} 

\newcommand{\forces}{\Vdash}
\newcommand{\wh}{\, \widehat{}\ }

\newtheorem{cl}{Claim}

\title{Oscillations and their applications in partition calculus}

\author{Laura Fontanella and Boban Veli\v{c}kovi\'c}

\institute{Equipe de Logique, Universit\'e de Paris 7,
2 Place Jussieu, 75251 Paris, France
}

\begin{document}

\maketitle
\thispagestyle{plain}

\pagestyle{plain}
\pagenumbering{arabic}

\begin{abstract}{Oscillations are a powerful tool for building examples of
colorings witnessing negative partition relations. We survey several results
illustrating the general technique and present a number of applications.}
\end{abstract}

\section{Introduction}

We start by recalling some well known notation.
Given three cardinals $\kappa, \lambda, \mu$ and $n<\omega,$ the notation
$$\kappa \to (\lambda)^n_{\mu}$$
means that for all function $f:[\kappa]^n\to \mu,$ there exists $H\subseteq \kappa$
with $\vert H\vert=\lambda$ and such that $f\restr [H]^n$ is constant.
We say that $f$ is a \emph{coloring of $[\kappa]^n$ in $\mu$ colors} and $H$ is an \emph{homogeneous set}.
Given $\kappa, \lambda, \mu, \sigma$ and $n$ as before, we write
$$
\kappa \to [\lambda]^n_{\mu}
$$
if for every coloring $f:[\kappa]^n\to \mu$ there exists $H\subseteq \kappa$ of
cardinality $\lambda$ such that $f''[H]^n\neq \mu$.
We write
$$
\kappa \to [\lambda]^n_{\mu, \sigma}
$$
if for every coloring $f:[\kappa]^n\to \mu$ there exists $H\subseteq \kappa$ of
cardinality $\lambda$ such that $|f''[H]^n|\leq \sigma$.

One can extend the above notation to sets with additional structures, such as linear or partial orderings,
graphs, trees, topological or vector spaces, etc. For instance, given two
topological spaces $X$ and $Y$, then
$$
X\to (\mathrm{top}\textrm{ }Y)^n_{\mu}
$$
means that for all $f:[X]^n\to \mu$ there exists a subset $H$ of $X$ homeomorphic to $Y$
such that $f\restriction [H]^n$ is constant.
Similarly, we can define statements such as $X\to [\mathrm{top}\textrm{ }Y]^n_{\mu}$,
$X\to [\mathrm{top}\textrm{ }Y]^n_{\mu,\sigma}$, etc.

We denote by $[\mathbb{N}]^{<\omega}$ the set of all finite subsets of $\mathbb{N}$
and by $[\mathbb{N}]^{\omega}$ the set of all infinite subsets of $\mathbb{N}$.
We often identify a set $s$ in $[\mathbb{N}]^{<\omega}$ (or $[\mathbb{N}]^{\omega}$)
with its increasing enumeration. When we do this, we will write
$s(i)$ for the $i$-th element of $s$, assuming it exists. In this way, we
identify $[\mathbb{N}]^{\omega}$ with $(\omega)^{\omega}$, the set of strictly increasing sequences
from $\omega$ to $\omega$, which is a $G_\delta$ subset of the Baire space $\omega^{\omega}$,
and thus is itself a Polish space. For $s,t\in [\mathbb{N}]^{<\omega}$
we write $s \sqsubseteq t$ to say that $s$ is an initial segment of $t$.
In this way, we can view $([\mathbb{N}]^{<\omega}, \sqsubseteq)$ as a tree.
For a given $s\in [\mathbb{N}]^{<\omega}$ we denote by $N_s$ the set of all infinite
increasing sequences of integers which extend $s$.
In general, if $T$ is a subtree of $[\mathbb{N}]^{<\omega}$ then $T_s$ will denote the set of
all sequences of $T$ extending $s$.
We will need some basic properties of the Baire space (or rather $[\mathbb{N}]^{\omega}$)
and the Cantor space $\{ 0,1\}^{\omega}$ with the usual product topologies.
For these facts and all undefined notions, we refer the reader to \cite{Kechris}.

The paper is organized as follows. In \S 2 we discuss partitions of the rationals
as a topological space. The basic tool is oscillations of finite sets of integers.
In \S 3 we consider infinite oscillations of tuples of real numbers and discuss several
applications to the study of inner models of set theory. In \S 4 we discuss
finite oscillations of tuples of reals of a slightly different type.
Finally, in \S 5 we present oscillations of pairs of countable ordinals and,
in particular, outline Moore's ZFC construction of an $L$-space.
We point out that none of the results of this paper are new and
we will give a reference to the original paper for each of the results we mention.
Our goal is not to give a comprehensive survey of all applications of oscillations
in combinatorial set theory, but rather to present several representative results
which illustrate the general method.

These are lecture notes of a tutorial given by the second author
at the 2nd Young Set Theory Workshop held at the CRM in Bellaterra,
April 14-18 2009. The notes were taken by the first author.

\section{Negative Partition Relations on the Rationals}

We start with a simple case of oscillations. Given $s,t\in [\mathbb{N}]^{<\omega}$
we define an equivalence relation $\sim$ on $s\bigtriangleup t$ by:

$$
n \sim m \iff ([n,m]\subseteq (s\setminus t)\lor [n,m] \subseteq (t\setminus s)),
$$
for all $n\leq m$ in $s\sdiff t.$ We now define the function $\osc: ([\mathbb{N}]^{<\omega})^2\to {\mathbb{N}}$ by
$$
\osc(s,t)=\vert (s\sdiff t)/_{\sim}\vert,
$$
for all $s,t\in [\mathbb{N}]^{<\omega}.$ If, for example, $s$ and $t$ are the two sets represented in the following picture,
then $\osc(s,t)=4.$

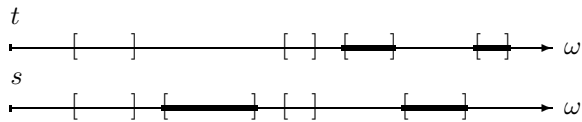
$$\setlength{\unitlength}{0.8cm}
\begin{picture}(10,3)
\put(0,1){\vector(1,0){9}}
\put(0,2){\vector(1,0){9}}
\put(0,0.95){\line(0,1){0.1}}
\put(0,1.95){\line(0,1){0.1}}
\put(9.2,0.9){$\omega$}
\put(9.2,1.9){$\omega$}
\linethickness{0.6mm}
\put(1,0.92){$[$}
\put(2,0.92){$]$}
\put(2.5,0.92){$[$}
\put(4,0.92){$]$}
\put(4.5,0.92){$[$}
\put(5,0.92){$]$}
\put(6.5,0.92){$[$}
\put(7.5,0.92){$]$}
\put(1,1.92){$[$}
\put(2,1.92){$]$}
\put(4.5,1.92){$[$}
\put(5,1.92){$]$}
\put(5.5,1.92){$[$}
\put(6.3,1.92){$]$}
\put(7.7,1.92){$[$}
\put(8.2,1.92){$]$}
\put(0,1.4){$s$}
\put(0,2.4){$t$}
\put(2.5,1){\line(1,0){1.6}}
\put(6.5,1){\line(1,0){1.1}}
\put(5.5,2){\line(1,0){0.9}}
\put(7.7,2){\line(1,0){0.6}}
\end{picture}$$


The following theorem is due to Baumgartner (see \cite{Baumgartner}).

\begin{theorem}[\cite{Baumgartner}]\label{Baum} $\mathbb{Q} \not\to [\mathrm{top}\spazio \mathbb{Q}]^2_{\omega}.$
\end{theorem}

This means, with our notation, that there exists a coloring $c: [\mathbb{Q}]^2 \to \omega$ such that
$c''[A]^2=\omega,$ for all $A\subseteq \mathbb{Q}$ with $A\approx \mathbb{Q}.$

Consider $[\mathbb{N}]^{<\omega}$ with the topology of pointwise convergence.
Let $X\subseteq [\mathbb{N}]^{<\omega}$ and $s\in [\mathbb{N}]^{<\omega}$.
Then $s\in \overline{X}$ iff for every $n >\sup(s)$ there is $t\in X$ such that $t\cap n =s$.
Given $s,t \in [\mathbb{N}]^{<\omega}$ we will write $s <t$ if $\max s<\min t$.

\begin{remark}
It is well known that $\mathbb{Q}\simeq [\mathbb{N}]^{<\omega},$ so we can view $\osc$ as a coloring of $[\mathbb{Q}]^2$.
\end{remark}

In order to prove Theorem \ref{Baum}, we recall the definition of the Cantor-Bendixson derivative:
$$\delta(X)=\{x\in X : \spazio x\in \overline{X\setminus \{x\}}\},$$
$$\delta^0(X)=X,$$
$$\delta^{k+1}(X)=\delta(\delta^{k}(X)).$$

We need the following lemma.

\begin{lemma}\label{lemmabaum} Suppose $X\subseteq [\mathbb{N}]^{<\omega}$ and $k>0$
is an integer such that $\delta^(k)\neq \emptyset$. Then $\osc''[X]^2\supseteq \{1, 2, ...,2k-1\}$.
\end{lemma}

\begin{proof} The proof is by induction on $k.$ First assume $k=1,$ then let $s\in \delta(X)$.
This means that we can find $t,u\in X$ such that $s\sqsubset t, u$ and $t\setminus s <u\setminus s$.
 It follows that $\osc(t,u)=1.$ Assume that the property holds for all $l<k,$ we show that $\osc''[X]^2$
 takes values $2k-2, 2k-1.$ Fix $s\in \delta^k (X).$ Recursively pick $u_i, t_i \in \delta^{k-i}(X)$, for all $i\leq k$,
 such that the following hold:

\begin{enumerate}
\item $t_0=u_0=s;$
\item $s\sqsubset t_1 \sqsubset t_2\cdots \sqsubset t_k;$
\item $s\sqsubset u_1 \sqsubset u_2\cdots \sqsubset u_k;$
\item $t_i\setminus t_{i-1} < u_i\setminus u_{i-1} < t_{i+1}\setminus t_i,$ for all $i\in \{1,2,...,k\}.$
\end{enumerate}

Then $\osc(t_{k-1}, u_{k-1})=2k-2$ and $\osc(t_{k}, u_{k-1})=2k-1.$  \qed
\end{proof}

\begin{proof}(of {\bf Theorem \ref{Baum}}).
By Remark 1, it is sufficient to check that for all $A\subseteq [\mathbb{N}]^{<\omega}$ homeomorphic
to $\mathbb{Q},$ $\osc''[A]^2=\omega.$ Since $A\approx \mathbb{Q},$ we have $\delta^k(A)\neq \emptyset$, for all
integers $k$. Hence we can apply Lemma \ref{lemmabaum} and this completes the proof. \qed
\end{proof}

 An unpublished result of Galvin states that
$$\eta\to [\eta]^2_{n,2}$$
when $\eta$ is the order type of the rational numbers, and $n$ is any integer. Therefore, the order theoretic
version of Theorem \ref{Baum} does not hold.
Also, the coloring we build to prove Baumgartner's theorem is not continuous.
In fact, if we only consider continuous colorings, then we have
$$\mathbb{Q}\to_{cont}[\mathrm{top}\spazio \mathbb{Q}]^2_2.$$
If we want a continuous coloring, we need to work in $[\mathbb{Q}]^3$. The following result
is due to Todor\v{c}evi\'{c} (\cite{Todorcevic_cube}).

\begin{theorem}[\cite{Todorcevic_cube}]\label{cont} There is a continuous coloring
$c: [\mathbb{Q}]^3 \to \omega$ such that $c''[A]^3=\omega$,
 for all $A\subseteq \mathbb{Q}$ with $A\approx \mathbb{Q}$.
\end{theorem}

\begin{proof}

Given $s,t,u \in [\mathbb{N}]^{<\omega},$ we define
$$\sdiff(s,t)=\min (s\sdiff t)$$
$$\sdiff(s,t,u)=\max \{\sdiff(s,t), \sdiff(t,u), \sdiff(s,u)\}.$$
The value of $\sdiff(s,t,u)$ is equal to the least $n\in \mathbb{N}$ such that
$\vert \{s\cap (n+1), t\cap (n+1), u\cap (n+1)\}\vert=3.$
So, in particular, for such an integer $n,$ we have
$\vert \{s\cap n, t\cap n, u\cap n\}\vert=2.$ Let $\{v, w\}=\{s\cap n, t\cap n, u\cap n\},$ then we define
$$\osc_3(s,t,u)=\osc(v,w).$$

The coloring \emph{$\osc_3$} is obviously continuous.
The proof that this coloring works is similar to the one given for Theorem \ref{Baum}.
We can prove, analogously, that if $X\subseteq [\mathbb{N}]^{<\omega}$ and
$\delta^{k}(X)\neq \emptyset$, for some integer $k>0$,
then $\osc_3''[X]^2\supseteq \{1, 2, ...,2k-1\}$.
Let us just see the case $k=1$. Fixing $s\in \delta(X)$ we can find
$t,u \in X$ such that $s\sqsubset t,u$ and $t\setminus s < u \setminus s$.
Then $\osc_3(s,t,u)=1$.
Finally one can apply this result to all subsets of $[\mathbb{N}]^{<\omega}$
that are homeomorphic to $\mathbb{Q},$
and this completes the proof. \qed
\end{proof}

\section{Oscillations of Real Numbers - Part 1}\label{part1}

We now discuss infinite oscillations and their applications.

Let $x\subseteq \mathbb{N},$ we define an equivalence relation $\sim_x$ on $\mathbb{N}\setminus x:$
$$n\sim_x m \iff [n,m]\cap x=\emptyset,$$
for all $n\leq m$ in $\mathbb{N}\setminus x.$ Thus, the equivalence classes of $\sim_x$ are
the intervals between consecutive elements of $x.$ Given $x,y,z\subseteq \mathbb{N},$
suppose that $(I_k)_{k\leq t}$ for $t\leq \omega$ is the natural enumeration of those equivalence classes
of $x$ which meet both $y$ and $z.$ We define a function $o(x,y,z):t\to \{0,1\}$ as follows:
$$o(x,y,z)(k)=0 \iff \min (I_{k}\cap y)\leq \min (I_{k}\cap z).$$


Notice that $o$ is a continuous function from
$$
\{(x,y,z)\in [[\mathbb{N}]^{\omega}]^3:  \vert (\mathbb{N}\setminus x)/_{\sim_x} \vert=\aleph_0\}
$$
to $2^{\leq \omega}$.

Note that $[\mathbb{N}]^{<\omega}$ ordered by $\sqsubseteq$ is a tree. A subset $T$
of $[\mathbb{N}]^{<\omega}$ is a {\em subtree} if it is closed under initial segments.

\begin{definition} Let $T$ be a subtree of $[\mathbb{N}]^{<\omega}$.
 We say that $t\in T$ is \emph{$\infty$-splitting} if for all
$k$ there exists $u\in T$ such that  $t \sqsubseteq u$ and $u(|t|)>k$.
 \end{definition}

\begin{definition} A subtree $T$ of $[\mathbb{N}]^{<\omega}$ is \emph{superperfect}
if for all $s\in T$ there exists $t\in T$
such that $s\sqsubseteq t$ and $t$ is $\infty$-splitting in $T$.
\end{definition}

\begin{definition}$X\subseteq [\mathbb{N}]^{\omega}$ is \emph{superperfect} if there is a superperfect tree
$T\subseteq [\mathbb{N}]^{<\omega}$ such that
$X=[T]=\{A\in [\mathbb{N}]^{\omega}:  A\cap k\in T, \mbox{ for all } k \}.$
\end{definition}

The following theorem is due to Veli\v{c}kovi\'{c} and Woodin (\cite{Velickovic}).

\begin{theorem}[\cite{Velickovic}]\label{continuouscoloring} Let $X, Y, Z\subseteq [\mathbb{N}]^{\omega}$
be superperfect sets.  Then $o''[X\times Y\times Z]\supseteq 2^{\omega}.$
\end{theorem}

\begin{proof} Let $T_1,T_2,T_3\subseteq [\mathbb{N}]^{<\omega}$ be superperfect trees
such that $X=[T_1],$ $Y=[T_2]$ and $Z=[T_3].$
Given an $\alpha\in 2^{\omega}$,
we build sequences $\langle s_k\rangle_k,$ $\langle t_k\rangle_k,$ $\langle u_k\rangle_k,$
of nodes of $T_1,T_2$ and $T_3$, respectively, such that the following properties hold:

\begin{enumerate}
\item $s_0,t_0,u_0$ are the least $\infty$-splitting node of $T_1,T_2$ and $T_3,$ respectively;
\item $s_0\sqsubset s_1\sqsubset s_2\sqsubset \cdots s_k\sqsubset\cdots;$
\item $t_0\sqsubset t_1\sqsubset t_2\sqsubset \cdots t_k\sqsubset\cdots;$
\item $u_0\sqsubset u_1\sqsubset u_2\sqsubset \cdots u_k\sqsubset\cdots;$
\item $t_i\setminus t_{i-1}, u_i\setminus u_{i-1}< s_i\setminus s_{i-1};$
\item $t_i\setminus t_{i-1}<u_i\setminus u_{i-1},$ if $\alpha(i)=0$ and\\
$u_i\setminus u_{i-1}<t_i\setminus t_{i-1},$ if $\alpha(i)=1.$
\end{enumerate}

If $x=\underset{k<\omega}{\bigcup} s_k,$ $y=\underset{k<\omega}{\bigcup}t_k$ and
$z=\underset{k<\omega}{\bigcup}u_k,$ then $o(x,y,z)=\alpha$ and this completes the proof.
\qed
\end{proof}

\begin{corollary}[\cite{Velickovic}]\label{beta} If $X\subseteq [\mathbb{N}]^{\omega}$ is a superperfect set,
then $o''[X]^3\supseteq 2^{\omega}.$ \qed
\end{corollary}

We now apply the previous theorem to prove some results about reals of inner models of set theory.

\begin{theorem}[\cite{Velickovic}]
Let $V,W$ be models of set theory such that $W\subseteq V$.
If there is a superperfect set $X$ in $V$ such that $X\subseteq W$ then $\mathbb{R}^{W}=\mathbb{R}^{V}$.
\end{theorem}

\begin{proof} This is trivial by applying Corollary \ref{beta}.
\qed
\end{proof}

\medskip

\begin{question} Can we replace superperfect by perfect in the previous theorem?
\end{question}

Surprisingly, the answer depends on whether $\mathrm{CH}$ holds in the model $W,$ as it is asserted
in the following theorem due to Groszek and Slaman (see \cite{Groszek}).

\begin{theorem}[\cite{Groszek}]\label{alpha} Suppose that $W$ and $V$ are two models of set theory
such that $W\subseteq V$.  Assume that there is a perfect set $P$ in $V$ such that $P\subseteq W$.
If $\mathrm{CH}$ holds in $W,$ then $\mathbb{R}^{W}=\mathbb{R}^{V}.$
\end{theorem}

In order to prove this theorem, let us introduce the following notion.

\begin{definition} Given two models of set theory $W$ and $V$ such that $W\subseteq V$
we say that $(W,V)$ satisfies the
 \emph{countable covering property for the reals} if, for all $X$ in $V$ such that
$X\subseteq \mathbb{R}^{W}$ and $X$ is countable in $V$,
there is an $Y$ in $W$ such that $X\subseteq Y$ and $Y$ is countable in $W$.
\end{definition}

We prove first the following theorem.

\begin{theorem}\label{theta} Given two models of set theory $W$ and $V$ such that
$W\subseteq V,$ suppose that there is a perfect set $P$ in $V$ such that $P\subseteq W$.
If $(W,V)$ satisfies the countable covering property for the reals,
then $\mathbb{R}^{W}= \mathbb{R}^{V}$.
\end{theorem}

\begin{proof} Work in $V$ and fix a perfect subset $P$ of $(2^{\omega})^{W}.$
Let $X$ be a countable dense subset of $P$. By the countable covering property for
the reals we can cover $X$ by some set $D\in W$ such that $D$ is countable
in $W$, is a dense subset of $2^{\omega}$ and $D\cap P$ is dense in $P$.
In $W$, fix an enumeration
$\{ d_n ; n <\omega \}$ of $D$. For $x,y\in 2^\omega$ with $x\neq y$ let
$$
\sdiff (x,y)=\min \{ n : x(n)\neq y(n)\}.
$$
Given $x\in 2^{\omega}\setminus D$ first define a sequence $\langle k_x(n); n <\omega\rangle$
by induction as follows
$$
k_x(n)= \min \{ k : \sdiff(x,d_k) > \sdiff(x,d_{k_x(i)}), \mbox{ for all } i <n\}.
$$
Note that $k_x(0)=0$. Since $D$ is dense in $P$ and $x\in P\setminus D$ then
$k_x(n)$ is defined, for all $n$.
Now define $f:P\setminus D\to [\mathbb{N}]^{\omega}$ by setting
$$
f(x)(n)=\sdiff(x,d_{k_x(n)}).
$$
Clearly, $f$ is continuous and
$f(x)$ is a strictly increasing function, for all $x\in 2^{\omega}\setminus D$.
Since $D\in W$ then $f$ is coded in $W$. We can now prove that $f''[P\setminus D]$ is superperfect.
Let $T=\{f(x)\restr n : x\in P\setminus D\land n\in\omega\}$.
First note that $f''[P\setminus D]$ is closed, i.e. it is equal to $[T]$.
To see this, note that if $b\in [T]$, then for every $i$ there is
$x_i\in P\setminus D$ such that $b\restriction i =f(x_i)\restriction i$.
Since $P$ is compact, it follows that the sequence $(x_i)_i$ converges to some $x\in P$.
Note then $k_x(n)=k_{x_m}(n)$, for all $m >n$, in particular, $k_x(n)$ is defined, for all $n$.
It follows that $x\notin D$. Since $f(x)=b$ it follows that
$b\in f''[P\setminus D]$, as desired.

Now, we show that every node of $T$ is $\infty$-splitting.
Let $s\in T$ and suppose $n=|s|$. Then there is some $x\in P\setminus D$ such that
$s\sqsubseteq f(x)$. Therefore, $s(i)=\sdiff(x,d_{k_x(i)})$, for all $i<n$.
Let $l=k_x(n)$. Since $P$ is perfect we can find, for every $j\geq \sdiff(x,d_l)$, some
$x_j\in P\setminus D$ such that $\sdiff(x_j,d_l)\geq j$. It follows
that $f(x_j)\restriction n =s$ and $f(x_j)(n)\geq j$. This shows
that $s$ is $\infty$-splitting.

Since $P\subseteq W$ and $f$ is coded in $W$ we have $f''[P\setminus D]\subseteq W$,
 that is $W$ contains a superperfect set.
 By Corollary \ref{beta}, we have $\mathbb{R}^{W}=\mathbb{R}^{V}$ and this completes the proof.
\qed
\end{proof}

\begin{proof}(of {\bf Theorem \ref{alpha}}).
By the previous theorem, it is enough to prove that $(W,V)$ satisfies
the countable covering property for the reals. By assumption, $W$ satisfies $\mathrm{CH}$,
so we can fix in $W$ a well-ordering on $(\mathbb{R})^W$ of height $(\omega_1)^W$.
Since every perfect set is uncountable and $P\subseteq W,$ then ${\omega_1}^W={\omega_1}^V$.
Therefore, any $X\subseteq (\mathbb{R})^W$ which is countable in $V,$ is contained in a proper
initial segment $Y$ of the well-ordering. Then $Y\in W$ and $Y$ is countable in $W.$ This completes the proof.
\qed
\end{proof}

In particular we can state the following corollary.

\begin{corollary}[\cite{Groszek}] If there is a perfect set of constructible reals,
then $\mathbb{R}\subseteq L$.
\qed
\end{corollary}

Is the countable covering condition necessary to obtain this result? Theorem \ref{gamma}
below (see \cite{Velickovic}) gives a partial answer to this question.

\begin{theorem}[\cite{Velickovic}]\label{gamma} There is a pair $(W,V)$ of generic extensions of $L$ with
$W\subseteq V$, such that $\aleph_1^{W}=\aleph_1^{V}$ and $V$ contains a perfect set of $W$-reals,
 but $\mathbb{R}^{W}\neq \mathbb{R}^{V}$.
\end{theorem}

On the other hand, in \cite{Velickovic} we have also the following theorem.

\begin{theorem}[\cite{Velickovic}]\label{delta} Suppose that $M$ is an inner model
of set theory and $\mathbb{R}^{M}$ is analytic,
then either $\aleph_1^{M}$ is countable, or all reals are in $M$.
\end{theorem}

In order to prove Theorem \ref{delta}, let us introduce a generalization of the notion of a superperfect set.

\begin{definition} Suppose $\lambda$ is a limit ordinal and $T$
is a subtree of $[\lambda]^{<\omega}$.
We say that $t\in T$ is \emph{$\lambda$-splitting} if for all $\xi <\lambda$
 there exists $u\in T$ such that $t\sqsubseteq u$
and $u(|t|)>\xi$.
\end{definition}

\begin{definition} Suppose $\lambda$ is a limit ordinal and let $T$ be
a subtree of $[\lambda]^{<\omega}$.
We say $T$ is \emph{$\lambda$-superperfect} if for all $s\in T$ there exists
$t\in T$ such that $s\sqsubseteq t$ and $t$ is $\lambda$-splitting.
\end{definition}

\begin{definition}\label{superperfect} A set $P\subseteq [\lambda]^{\omega}$ is \emph{$\lambda$-superperfect}
if there is a $\lambda$-superperfect tree $T\subseteq [\lambda]^{<\omega}$
such that $P=\{x\in [\lambda]^{\omega} : \forall n<\omega (x\restriction n \in T)\}$.
Here $x\restriction n$ denotes the set of the first $n$ elements of $x$ in the natural order.
\end{definition}

The definition of $o:([\mathbb{N}]^\omega)^3\to \{0,1\}^{\omega}$ can be trivially generalized
to a coloring
$$
o_{\lambda}:([\lambda]^{\omega})^3\to \{0,1\}^{\omega}.
$$
As for $o$ one can easily check that for all $\lambda$-superperfect $P$,
we have $o_{\lambda}''[P^3]\supseteq \{0,1\}^{\omega}$
(the proof is the same as for Theorem \ref{continuouscoloring}).
Moreover, we have $o_{\lambda}(x,y,z)\in L[x,y,z],$ for all $x,y,z\in [\lambda]^{\omega}$.
To complete the proof of Theorem \ref{delta} it suffices
to prove the following lemma.

\begin{lemma}\label{epsilon} Suppose that $A$ is an analytic set such that
$\sup \{ \omega _1^{{\rm CK},x}: x\in A \}=\omega _1$. Then every real
is hyperarithmetic in a quadruple of elements of $A$.
\end{lemma}

\begin{proof} Let $T \subset (\omega \times \omega)^{<\omega}$
be a tree such that $A=p[T]$. Note that the statement
 sup$\{ \omega _1^{{\rm CK},x}:x\in p[T]\}=\omega_1$ is $\Pi _2^1(T)$ and
thus absolute.

For an ordinal $\alpha$ let $Coll(\aleph _0,\alpha)$
be the usual collapse of $\alpha$ to $\aleph _0$ using finite
conditions.
Let ${\cal P}$ denote $Coll(\aleph _0,\aleph _1)$.
 If $G$ is $V$-generic over ${\cal P}$,
by Shoenfield's absoluteness theorem,
in $V[G]$ there is $x \in p[T]$ such that $\omega _1^{{\rm CK},x}>\omega _1^V$.
In $V$ fix a name $\dot{x}$ for $x$ and a name $\sigma$ for a cofinal
$\omega$-sequence in $\omega _1^V$ such that the maximal condition
in ${\cal P}$ forces that $\dot{x}\in p[T]$ and $\sigma \in L[\dot{x}]$.


\begin{cl}\label{one} For every $p \in {\cal P}$ there is $k<\omega$
such that for every $\alpha <\omega _1$ there is $q\leq p$ such that
$q\forces \sigma (k)>\alpha$.
\end{cl}

\begin{proof} Assume otherwise and fix $p$ for which the claim is false.
Then for every $k$ there is $\alpha _k<\omega _1$ such that
$p \forces \sigma (k)<\alpha _k$. Let $\alpha =\sup \{ \alpha _k:k<\omega \}$.
Then $p\forces \mbox{ran}(\sigma)\subset \alpha$, contradicting the fact
that $\sigma$ is forced to be cofinal in $\omega _1^V$.
\qed
\end{proof}

Let ${\cal Q}$ denote $Coll(\aleph _0,\aleph _2)$ as defined in $V$.
 Suppose $H$ is $V$-generic over ${\cal Q}$.
Work for a moment in $V[H]$. If $G$ is a $V$-generic filter over ${\cal P}$
let $\sigma _G$ denote the interpretation of $\sigma$ in $V[G]$.
Let $B$ be the set of all $\sigma _G$ where $G$ ranges over all $V$-generic
filters over ${\cal P}$.

\medskip
\begin{cl}\label{two} $B$ contains an $\omega _1^V$-superperfect set
 in $(\omega _1^V)^{\omega}$.
\end{cl}

\begin{proof} Let $\{ D_n:n<\omega \}$ be an enumeration of all dense subsets
of ${\cal P}$ which belong to the ground model.   For each
 $t\in (\omega _1^V)^{<\omega}$
we define a condition $p_t$ in the regular open algebra of ${\cal P}$
as computed in $V$ and $s_t\in (\omega _1^V)^{<\omega}$ inductively on the length
of $t$ such that

\begin{enumerate}
\item $p_t\in D_{lh(t)}$
\item $p_t \forces s_t \subset \sigma$
\item if $t\subseteq r$ then $p_r\leq p_t$ and $s_t \subset s_r$
\item if $t$ and $r$ are incomparable then $s_t$ and $s_r$ are incomparable
\item for every $t$ the set $\{ \alpha <\omega_1^V: \mbox{there is}\ q \leq p\ q\forces s_t\wh \alpha \subset \sigma\}$
is unbounded in $\omega _1^V$.
\end{enumerate}

Suppose $p_t$ and $s_t$ have been defined. Using 5. choose in $V$
a 1-1 order preserving function $f_t:\omega _1^V\rightarrow \omega _1^V$
and for every $\alpha$  $q_{t,\alpha}\leq p_t$ such that
 $q_{t,\alpha}\forces s_t\wh f_t({\alpha})\subset \sigma$.
By extending $q_{t,\alpha}$ if necessary, we may assume that it belongs to $D_{lh(t)+1}$.
Now, by applying Claim 1, we can find a condition $p\leq	q_{t,\alpha}$
and $k>lh(s_t)+1$ such that for some $s\in (\omega _1^V)^k$\
$p\forces s\subset \sigma$ and for every  $\gamma <\omega _1^V$
there is $q\leq p$ such that $q\forces \sigma (k)>\gamma$.
Let then $s_{t\  \wh \alpha}=s$ and $p_{t\ \wh \alpha}=p$.
This completes the inductive construction.

Now if $b\in (\omega _1^V)^{\omega}$ then $\{ p_{b\restriction n}:n<\omega\}$
generates a filter $G_b$ which is $V$-generic over ${\cal P}$.
The interpretation of $\sigma$ under $G_b$ is $s_b=\bigcup _{n<\omega}s_{b\restriction n}$.
Since the set $R=\{ s_b:b\in (\omega _1^V)^{\omega}\}$ is $\omega _1^V$-superperfect,
this proves Claim \ref{two}.
\qed
\end{proof}

Now, using the remark following Definition \ref{superperfect}, for any real $r\in \{ 0,1\}^{\omega}$
we can find $b_1,b_2,b_3 \in (\omega _1^V)^{\omega}$
 such that $r \in L[s_{b_1},s_{b_2},s_{b_3}]$.
Let $x_i$ be the interpretation of $\dot{x}$ under $G_{b_i}$.
Then it follows that $x_i\in p[T]$ and $s_{b_i}\in L[x_i]$, for $i=1,2,3$.
Thus $r\in L[x_1,x_2,x_3]$.
Pick a countable ordinal $\delta$ such that $r\in L_{\delta}[x_1,x_2,x_3]$.
Using the fact that in $V[H]$
$\sup \{ \omega _1^{{\rm CK},x}:x\in p[T]\} =\omega _1$,
we can find $y\in p[T]$ such that $\omega _1^{{\rm CK},y}>\delta$.
Then we have that $r$ is $\Delta _1^1(x_1,x_2,x_3,y)$.
Note that the statement that there are $x_1,x_2,x_3,y \in p[T]$ such that
$r\in \Delta _1^1(x_1,x_2,x_3,y)$ is $\Sigma _2^1(r,T)$. Thus for any real $r \in V$,
by Shoenfield absoluteness again, it must be true in $V$. This proves Lemma \ref{epsilon}.
\qed
\end{proof}

We complete this section by stating some related results.

\begin{theorem}[\cite{Velickovic}] There is a pair of generic extensions of
$L,$ $W\subseteq V$ such that $\mathbb{R}^W$ is an uncountable $F_{\delta}$
set in $V,$ and $\mathbb{R}^W\neq \mathbb{R}^V$.
\end{theorem}


\begin{theorem}[\cite{Gitik}]\label{Gitik}
Suppose that $W\subseteq V$ are two models of set theory, $\kappa>\omega_1^V$
and there exists  $C\subseteq [\kappa]^{\omega}$ which is a club in $V$ such that $C\subseteq W$.
 Then $\mathbb{R}^W=\mathbb{R}^V.$
\end{theorem}

\begin{theorem}[\cite{Caicedo}]\label{Caicedo}
 Suppose that $W\subseteq V$ are two models of set theory such that
 $V,W\models {\rm PFA}$ and $\aleph_2^W=\aleph_2^V.$
 Then $\mathbb{R}^W=\mathbb{R}^V,$ in fact $\pow{\omega_1}^W=\pow{\omega_1}^V.$
\end{theorem}

\section{Oscillations of Real Numbers - part 2}

The results of this section are taken from \cite{Todorcevic}.
We look increasing sequences of integers and slightly change the definition of oscillation.
Given $s,t\in (\omega)^{\leq \omega}$ we define
$$
\osc(s,t)=\vert \{n<\omega :\spazio s(n)\leq t(n) \land s(n+1)>t(n+1)\}\vert.
$$

In the next picture, $s$ and $t$ are two functions in $(\omega)^{<\omega}$ with $\osc(s,t)=2.$\\

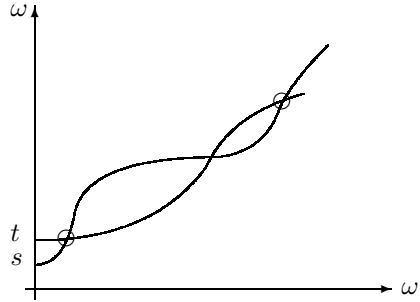
$$\setlength{\unitlength}{0.65cm}
\begin{picture}(8,7)
\put(-0.2,1){\vector(1,0){7.5}}
\put(0,0.8){\vector(0,1){6}}
\put(7.5,0.9){$\omega$}
\put(-0.5,6.6){$\omega$}
\put(-0.5,1.5){$s$}
\put(-0.5,2){$t$}
\qbezier(0,1.5)(0.6,1.5)(0.8,2.5)
\qbezier(0.8,2.5)(1,3.7)(3.7,3.7)
\qbezier(3.7,3.7)(4.7, 3.8)(5,4.7)
\qbezier(5,4.7)(5.25,5.25)(6,6)
\qbezier(0,2)(2.4, 2)(3.5, 3.5)
\qbezier(3.5,3.5)(4,4.6)(5.5,5)
\put(0.64,2.05){\circle{0.3}}
\put(5.05,4.85){\circle{0.3}}
\end{picture}
$$
We now define two orders $\leq_m$ and $\leq_*$ on $(\omega)^{\omega}$:

$$
f\leq_m g\iff \forall n\geq m( f(n)\leq g(n));
$$
$$
f\leq_* g \iff \exists m (f\leq_m g).
$$
Given $X\subseteq (\omega)^{\omega}$ and $s\in (\omega)^{<\omega}$ we let
$X_s=\{f\in X : s\sqsubseteq f\}$ and
$$
T_X=\{s\in (\omega)^{<\omega}: X_s \mbox{ is unbounded under }\leq_*\}.
$$

\begin{lemma}\label{claim} Suppose that $X\subseteq (\omega)^{\omega}$ is unbounded under
$\leq_*$ and $X=\underset{n<\omega}{\bigcup} A_n$. Then there exists $n$ such that $A_n$ is unbounded.
\end{lemma}

\begin{proof} Suppose that every $A_n$ is bounded and, for all $n,$ let $g_n$
be such that $f\leq_* g_n,$ for all $f\in A_n.$ If we define
$g(n)=\sup \{g_k(n) : k\leq n\},$ then $X$ is bounded by $g$ with respect to $\leq_*$.
This leads to a contradiction. \qed
\end{proof}

\begin{lemma}\label{aaaaa} Suppose that $X\subseteq (\omega)^{\omega}$ is unbounded under $\leq_*$.
Then $T_X$ is superperfect.
\end{lemma}

\begin{proof} Suppose, by way of contradiction, that there is a node $s\in T_X$ with
no $\infty$-splitting extensions in $T_X$. We define a function $g_s:\omega\to\omega$ as follows:

$$
g_s(n)=\sup \{t(n):  t\in (T_X)_s\land n\in dom(t)\},
$$
where $(T_X)_s=\{f\in T_X:  s\sqsubseteq f\}$.
First note that $g_s(n)<\omega$, for all $n<\omega$.
Let $Q=\{ t\in (\omega)^{<\omega} : X_t \mbox{ is bounded under } \leq_*\}$.
By Lemma \ref{claim}, $\bigcup \{ X_t: t\in Q\}$ is bounded under $\leq_*$ by
some function $g$. Now let $h=\max (g,g_s)$. It follows that $X_s$
is $\leq_*$-bounded by $h$, a contradiction.
\qed
\end{proof}


We first consider oscillations of elements of $(\omega)^{<\omega}$.
Our first goal is to prove that if $T$
is a superperfect subtree of $(\omega)^{<\omega}$
then $\osc''[T]^2=\omega$. In fact, we prove
a slightly stronger lemma.

\begin{lemma}\label{estendereinodi} Let $S$ and $T$ be two superperfect subtrees of
$(\omega)^{<\omega}$ and let $s$ and $t$ be $\infty$-splitting nodes of $S$ and $T$
respectively. Then for all $n$ there are  $\infty$-splitting nodes $s'$ in $S$ and
$t'$ in $T$ such that $s\sqsubseteq s',$ $t\sqsubseteq t'$ and
$$\osc(s',t')=\osc(s,t)+n.$$
\end{lemma}

\begin{proof} We may assume without
loss of generality that $\vert s\vert<\vert t\vert$ and
$s(\vert s\vert-1)\leq t(\vert s\vert -1).$
We can recursively pick some $\infty$-splitting extensions
$s_i\in S$ and $t_i\in T$, for $i\leq n$ such that:

\begin{itemize}
\item $s_0=s$ and $t_0=t$;
\item $s_0\sqsubset s_1\sqsubset \cdots \sqsubset s_n$;
\item $t_0\sqsubset t_1\sqsubset \cdots \sqsubset t_n$;
\item $\osc(s_i, t_i)=\osc(s,t)+i,$ for all $i;$
\item $\vert s_i\vert<\vert t_i\vert$ and $s_i(\vert s_i\vert-1)\leq t_i(\vert s_i\vert-1)$.
\end{itemize}

Given $s_i$ and $t_i$, since $S$ is superperfect and $s_i$ is $\infty$-splitting
in $S$, we can find some $\infty$-splitting extension $u$ of $s_i$ in $S$
such that $u(\vert s_i\vert)>t_i(\vert t_i\vert-1)$ and such that $\vert u\vert>\vert t_i\vert +1$.
In the same way, we can take an $\infty$-splitting extension $v$ of $t_i$ in $T$
such that $v(\vert t_i\vert)>u(\vert u\vert-1)$ and $\vert v\vert>\vert u\vert +1$.
Since $u$ and $v$ are strictly increasing, we have $\osc(u,v)=\osc(s_i,t_i)+1$,
so we can define $s_{i+1}=u$ and $t_{i+1}=v.$

Finally, $\osc(s_n, t_n)=osc(s,t)+n$ and this completes the proof.
\qed
\end{proof}

\begin{corollary}\label{cor} If $T$ is a superperfect subtree of $(\omega)^{<\omega}$
then $\osc''[T]^2=\omega.$
\qed
\end{corollary}

We now turn to oscillations of elements of $(\omega)^{\omega}$.
We will need the following definition.

\begin{definition}A subset $X$ of $(\omega)^{\omega}$
is  $\sigma$-directed under $\leq_*$ if, and only if,
for all countable $D\subseteq X$ there is $f\in X$ such that $d\leq_* f$,
for all $d\in D.$ \end{definition}

\begin{lemma}\label{estendereinXY} Suppose $X\subseteq (\omega)^{\omega}$
and $\sigma$-directed and unbounded under $\leq_*$ and $Y\subseteq (\omega)^{\omega}$
is such that for every $a\in X$ there is $b\in Y$ such that $a\leq_*b$.
There is an integer $n_0$ such that for all $k<\omega$
there is $f\in X$ and $g\in Y$ such that $\osc(f,g)=n_0+k.$
\end{lemma}

\begin{proof} Fix a countable dense subset $D$ of $X$.
Since $X$ is $\sigma$-directed, there is a function $a\in X$ such that
$d\leq_* a$, for all $d\in D$.
The set $Y'=\{g\in Y:  a\leq_* g \}$ is unbounded under $\leq_*$.
We define $Y_m=\{g\in Y':  a\leq_m g\}$, for all $m<\omega$.
By Lemma \ref{claim} and the fact that $Y'=\bigcup \{ Y_m: m<\omega\}$,
there exists $m_0<\omega$ such that $Y_{m_0}$ is also $\leq_*$-unbounded.
Let $s_0\in T_X$ and $t_0\in T_{Y_{m_0}}$ be the two least
$\infty$-splitting nodes of $T_X$ and $T_{Y_{m_0}}$ respectively.
Let $n_0=\osc(s_0,t_0)$.
Now, fix $k<\omega$. By Lemma \ref{estendereinodi}, there are two $\infty$-splitting $s\in T_X$
and $t\in T_{Y_{m_0}}$  such that $\osc(s,t)=n_0+k$.
We may assume without loss of generality that $\vert t\vert\leq \vert s\vert$ and
$t(\vert t\vert-1)>s(\vert t\vert -1).$ Since $D$ is dense, there is $f\in D$
such that $s\sqsubseteq f\leq_* a$ Fix $m\geq m_0$ such that $f\leq_m a$.
Since $t$ is $\infty$-splitting in $T_{Y_{m_0}}$, we can pick $i>f(m)$ and $g\in Y_{m_0}$,
such that $t\, \wh i\sqsubseteq g$. We know that for all $k\geq m_0$, $a(k)\leq g(k)$,
so for all $k\geq m,$ $f(k)\leq g(k)$.
Moreover, $f$ and $g$ are increasing and $t\, \wh i \sqsubseteq g$,
so for all $k$ between $\vert t\vert$ and $m$ we have $g(k)>f(k)$.
It follows that $\osc(f,g)=\osc(s,t)=n_0+k$ and this completes the proof.
\qed
\end{proof}

The following theorem is due to S. Todor\v{c}evi\'c (see \cite{Todorcevic}).

\begin{theorem}[\cite{Todorcevic}] Suppose $X\subseteq (\omega)^{\omega}$ is unbounded under $\leq_*$
and $\sigma$-directed, then $osc''[X]^2=\omega.$
\end{theorem}

\begin{proof} The proof is the same as for Lemma \ref{estendereinXY}, by assuming $Y=X$.
Thus $s_0=t_0$ and, consequently, $n_0=0$ in the previous proof.
Hence, for all $k<\omega$ there are $f,g\in X$ such that $\osc(f,g)=k.$ This completes the proof.
\qed
\end{proof}

We recall that $\mathfrak{b}$ is the least cardinal of an $\leq_*$-unbounded subset of
$(\omega)^{\omega}.$ Fix an unbounded $\mathscr{F}\subseteq (\omega)^{\omega}$ of
cardinality $\mathfrak{b}$. We may assume that
$\mathscr{F}$ is well ordered under $\leq_*$ and $(\mathscr{F},\leq_*)$ has order type $\mathfrak{b}$.

\begin{remark}\label{four} Every unbounded subset of $\mathscr{F}$ is $\sigma$-directed
and cofinal in $\mathscr{F}$ under $\leq_*$.
\end{remark}

\begin{corollary}\label{ggggg} Let $X,Y \subseteq \mathscr{F}$ be unbounded under $\leq_*$.
There exists $n_0< \omega$ such that for all $k<\omega$ there exist $f\in X$
and $g\in Y$ such that $\osc(f,g)=n_0+k$.
\end{corollary}

\begin{proof} Trivial by Remark \ref{four} and Lemma \ref{estendereinXY}.
\qed
\end{proof}

\noindent In \cite{Todorcevic}, Todor\v{c}evi\'c proved a more general result:

\begin{theorem}[\cite{Todorcevic}]\label{h}
Suppose $\mathscr{F}$ is $\leq_*$-unbounded and well ordered by $\leq_*$
in order type $\mathfrak{b}$. Suppose $\mathfrak{A}\subseteq [\mathscr{F}]^n$,
 $\vert\mathfrak{A}\vert=\mathfrak{b}$ and $\mathfrak{A}$ consists of
 pairwise disjoint $n$-tuples.
 Then there exists $h: n\times n\to \omega$ such that for all $k<\omega$
 there exist $A,B\in \mathfrak{A}$ such that $A\neq B$ and
 $\osc(A(i), B(j))=h(i,j)+k$, for all $i,j<n$.
 Here $A(i)$ denotes the i-th element of $A$ in increasing order,
 and similarly $B(j)$ denotes the $j$-th element of $B$.
\end{theorem}

\begin{proof} For any $A,B\in [\mathscr{F}]^n,$ we will write $A<_m B$ if,
and only if, $a<_m b$ for all $a\in A$ and $b\in B.$ Similarly, with $A\leq_* B$
we mean that  $a\leq_* b,$ for all $a\in A$ and $b\in B.$
Finally, if $A\in \mathfrak{A}$ and $m<\omega,$ we denote by
$A\restr m$ the sequence $\langle A(i)\restr m\rangle_{i<n}.$

We may assume that $\mathfrak{A}$ is \emph{everywhere unbounded},
that is for all $m<\omega$ and $A\in \mathfrak{A},$ the set
$\{B\in \mathfrak{A}:  B\restr m= A\restr m\}$ is also unbounded in $((\omega)^{\omega})^n$
under $\leq_*$. Take a countable dense $\mathfrak{D}\subseteq \mathfrak{A}$.
There is $A\in \mathfrak{A}$ such that $D\leq_* A,$ for all $D\in \mathfrak{D}$.
For all $m<\omega,$ let $\mathfrak{A}_m=\{B\in \mathfrak{A}:  A<_m B\}$.
As before, there is $m_0<\omega$ such that $\mathfrak{A}_{m_0}$ is everywhere unbounded.


Given any $\vec{t}\in (\omega^{<\omega})^n,$ we denote by $t_i$ the
$i$-th element of $\vec{t}$ in increasing order. If $B\in (\omega^{\omega})^n,$
then $\vec{t}\sqsubseteq B$ means $t_i\sqsubseteq B(i),$ for all $i<n.$ Now, we define

$$
T_{\mathfrak{A}_{m_0}}=\{\vec{t}\in (\omega^{<\omega})^n:
\forall i<n(\vert t_i\vert<\vert t_{i+1}\vert)\land \exists B\in \mathfrak{A}_{m_0}(\vec{t}\sqsubseteq B)\}.
$$

For any sequence $\vec{s}\in T_{\mathfrak{A}_{m_0}},$ we say that \emph{$\vec{s}$ is
$\infty$-splitting} if for all $l<\omega,$ there is $\vec{t}\in T_{\mathfrak{A}_{m_0}}$
such that $\vec{s}\sqsubseteq \vec{t}$ and $t_i(\vert s_i\vert)>l$, for all $i<n$.

\begin{cl}\label{mu} $T_{\mathfrak{A}_{m_0}}$ is superperfect, that is for all
$\vec{s}\in T_{\mathfrak{A}_{m_0}}$, there is an $\infty$-splitting sequence
$\vec{t}\in T_{\mathfrak{A}_{m_0}}$ such that $\vec{s}\sqsubseteq \vec{t}$.
\end{cl}

\begin{proof}
Given $\vec{s}\in T_{\mathfrak{A}_{m_0}},$ define $t_0$ as the least
$\infty$-splitting extension of $s_0$ in $T_{Z(0)},$ where
$Z(0):=\{B(0):  B\in \mathfrak{A}_{m_0}\}$. Assume that $\vec{t}\restr i$ is defined, the set
$Z(i):=\{ B(i); \spazio B\in \mathfrak{A}_{m_0}\textrm{ and }B\restr i=\vec{t}\restr i\}$
is unbounded (because $\mathfrak{A}_{m_0}$ is everywhere unbounded).
Put $t_i$ any $\infty$-splitting extension of $s_i$ in $T_{Z(i)},$
such that $\vert t_i\vert>\vert t_{i-1}\vert.$ The sequence $\vec{t}$,
so defined, is $\infty$-splitting in $T_{\mathfrak{A}_{m_0}}.$
This completes the proof of the claim.
\qed
\end{proof}

Now, let $\vec{r}\in T_{\mathfrak{A}_{m_0}}$ be the least $\infty$-splitting sequence, we define for all $i,j<n,$
$$
h(i,j)= \osc(r_i, r_j).
$$

\begin{cl}\label{nu} For all $k<\omega,$ there are two $\infty$-splitting sequences
$\vec{s}, \vec{t}\in T_{\mathfrak{A}_{m_0}}$ such that $\vec{r}\sqsubseteq \vec{s}, \vec{t}$ and
$osc(s_i, t_j)=osc(r_i, r_j)+k,$ for all $i,j<n$.
\end{cl}

\begin{proof} We prove this by induction on $k< \omega$.
The case $k=0$ is trivial. Let $\vec{s}, \vec{t}\in T_{\mathfrak{A}_{m_0}}$ be $\infty$-splitting,
such that  $\vec{r}\sqsubseteq \vec{s}, \vec{t}$ and $osc(s_i, t_j)=osc(r_i, r_j)+k$, for all $i,j$.
Assume without loss of generality that $\vert s_i\vert<\vert t_j\vert$ and
$s_i(\vert s_i\vert-1)\leq t_j(\vert s_i\vert-1),$ for all $i,j$.
Since $\vec{s}$ is $\infty$-splitting, there is an $\infty$-splitting sequence
$\vec{u}\in T_{\mathfrak{A}_{m_0}}$
such that $\vec{s}\sqsubseteq \vec{u}$ and $u_i(\vert s_i\vert)>t_j(\vert t_j\vert-1)$,
for all $i,j.$
We ask, also, for $\vert u_i\vert>\vert t_j\vert+1,$ for all $i,j$.
Similarly, we can find an $\infty$-splitting sequence $\vec{v}\in T_{\mathfrak{A}_{m_0}}$
such that $\vec{t}\sqsubseteq \vec{v}$ and $v_i(\vert t_i\vert)>u_j(\vert u_j\vert-1)$,
for all $i,j.$
It follows that $\osc(u_i, v_j)=osc(s_i, t_j)+1$ for all $i,j.$ This completes the proof of the claim.
\qed
\end{proof}

Fix $\vec{s}$ and $\vec{t}$ as in Claim \ref{nu}, assume without loss of generality
that $\vert s_i\vert\leq \vert t_j\vert$ and $s_i(\vert s_i\vert-1)>t_j(\vert s_i\vert-1)$,
for all $i,j$.
Consider, now, the families
$X=\{B\in \mathfrak{A}:  \vec{t}\sqsubseteq B\}$ and $\mathfrak{D}'=\mathfrak{D}\cap X$.
We have that $X$ is everywhere unbounded and $\mathfrak{D}'$ is dense in $X$.
Take any $D\in \mathfrak{D}',$ then $\vec{t}\sqsubseteq D<_m A$ for some $m>m_0$.
Since $\vec{s}$ is $\infty$-splitting there is $l\geq D(n-1)(m)$ and
$B\in \mathfrak{A}_{m_0}$ such that
$\vec{s}\, \wh l:=\langle s_i\, \wh l\rangle_{i<n}\sqsubseteq B$.
By construction, $\osc(D(i), B(j))=\osc(s_i, t_j)=h(i,j)+k,$ for all $i,j<n$.
This completes the proof.
\qed
\end{proof}


Sometimes we need to improve $\osc$ to get an even better coloring.
First we want to get rid of the function $h$ of Theorem \ref{h}.
We fix a bijection $\omega\overset{e}{\to}\omega\times\omega.$
We define a new partial function $o$ on pairs of elements of
$(\omega)^{< \omega}$ or $(\omega)^{\omega}$ as follows.
Suppose $\osc(f,g)=2^{i_0}+2^{i_1}+\cdots+2^{i_k}$ for $i_0>i_1>\cdots >i_k$,
be the binary expansion of $\osc(f,g).$
We define $o(f,g)=\pi_0\circ e(i_0)$ where $\pi_0$ is the projection on the first component.

\begin{lemma}\label{sette}
Suppose $\mathscr{F}$ and $\mathfrak{A}\subseteq [\mathscr{F}]^n$ are as
in Theorem \ref{h}. Then for all $k<\omega$ there exists $A,B\in \mathfrak{A}$
such that $A\neq B$ and $o(A(i), B(j))=k,$ for all $i,j<n$.
\end{lemma}

\begin{proof} Given $k$, consider the function $h: n\times n\to\omega$
of Theorem \ref{h}. For all $i,j<n,$ let $l_{i,j}$ be the largest integer
such that $2^{l_{i,j}}\leq h(i,j)$ and let $l=\max \{l_{i,j}; i,j<n\}$.
The set $\{m:  \exists p(e(m)=(k,p))\}$ is infinite so we can find $m>l$
such that $\pi_0\circ e(m)=k.$ By definition of $h$ there exist two different
$A,B\in \mathfrak{A}$ such that $\osc(A(i),B(j))=h(i,j)+2^m,$ for all $i,j<n$.
It follows that $o(A(i),B(j))=\pi_0\circ e(m)=k,$ for all $i,j<n.$ This completes the proof.
\qed
\end{proof}

Finally we want to be able to choose the color of $\{A(i), B(j)\}$
independently for all $i,j$. First we need the following lemma.

\begin{lemma}\label{these} Given $\mathfrak{A}\subseteq [\mathscr{F}]^n$ an unbounded family of parwise disjoint sets.
There are $k<\omega$ and $\mathfrak{A}^*\subseteq \mathfrak{A}$ unbounded such that, for all $i<n$,
there exists $a_i\in (\omega)^k$ such that $A(i)\restr k=a_i,$ for all $A\in \mathfrak{A}^*$ and
$a_i\neq a_j,$ for all $i\neq j<n.$
\end{lemma}

\begin{proof}We prove it by induction on $n<\omega$. It is trivial for $n=1.$ Assume that the statement
is true for $n,$ we prove it for $n+1.$ Given $\mathfrak{A}\subseteq [\mathscr{F}]^{n+1},$ let be
$k<\omega$, $\mathfrak{A}'\subseteq \mathfrak{A}\restr n,$ and $\{ a_i\}_{i<n}$ as in
the conclusion of the lemma for $\mathfrak{A}\restr n.$
The set $\mathfrak{B}=\{A\in \mathfrak{A}:  A\restr n\in \mathfrak{A}'\}$ is unbounded,
hence $X=\{A(n):  A\in \mathfrak{B}\}$ is also unbounded. By Lemma \ref{aaaaa}, we have that
$T_{X}$ is superperfect, so let $b$ be the least $\infty$-splitting node of $T_{X}.$ We can assume without
loss of generality that $\vert b\vert<k$.
Take any $a_n\sqsupseteq b$ in $T_X$ such that $\vert a_n\vert=k$ and $a_n(k -1)> \max \{ a_i(k -1):  i<n \}$.
Then $a_n\neq a_i,$ for all $i<n.$
Recall that $T_X=\{s\in (\omega)^{<\omega}:  \{f\in X:  s\sqsubseteq f\}\textrm{ is unbounded}\}$,
thus $\mathfrak{A}^*=\{B\in \mathfrak{B}:  a_n\sqsubseteq B(n)\}$ is unbounded. This completes the proof.
\qed
\end{proof}


Consider all finite functions $t: D\times E\to \omega$ where $D,E\subseteq (\omega)^k$ and $k$ is an integer.
Let $\{(t_n, D_n, E_n, k_n)\}_{n<\omega}$ be any enumeration of such functions. 
We define $c:[\mathscr{F}]^2\to \omega$ as follows: given $f,g\in \mathscr{F}$ and letting $n=o(f,g)$, we define
$$c(f,g)=\left\{ \begin{array}{ll}
t_n(f\restr k_n, g\restr k_n)&\textrm{if $f\restr k_n\in D_n$ and $g\restr k_n\in E_n$}\\
0&\textrm{otherwise}
\end{array}\right. $$

\begin{theorem}[\cite{Todorcevic}]
Given an unbounded family $\mathfrak{A}\subseteq [\mathscr{F}]^n$ of pairwise disjoint sets,
and an arbitrary $u: n\times n\to \omega,$ there are two different $A,B\in \mathfrak{A}$ such that $c(A(i), B(j))=u(i,j)$,
for all $i,j<n$.
\end{theorem}

\begin{proof} Take $k<\omega,$ $\mathfrak{A}^*$ and $\{a_i\}_{i<n}$ as in the conclusion of Lemma \ref{these}
and let $D=\{a_i:  i<n\}$. Consider the function $t:D\times D\to \omega$ defined by $t(a_i,a_j)=u(i,j)$, for all $i,j<n$.
 Assume that $(t_m, D_m, E_m, k_m)$ is the corresponding sequence in the previous enumeration.
 By Lemma \ref{sette} there exist different $A,B\in \mathfrak{A}^*$ such that $o(A(i),B(j))=m,$ for all $i,j<n.$
 It follows that $u(i,j)=t(a_i,a_j)=t_m(A(i)\restr k_m, B(j)\restr k_m)=c(A(i),B(j)).$ This completes the proof.
 \qed
\end{proof}

\begin{corollary} There is a $\mathfrak{b}-c.c.$ partial order whose square is not $\mathfrak{b}-c.c.$
\qed
\end{corollary}


The following question is still open.

\begin{question} Can we do the same for some other cardinal invariant such as $\mathfrak{t}$ or $\mathfrak{p}$?
\end{question}

\section{Partitions of countable ordinals}

Oscillations provide the main tool for constructing partitions of pairs of countable ordinals
with very strong properties. The goal of this section is to present the construction of an $L$-space
due to Moore \cite{Moore} which uses oscillations in an ingenious way.
In order to motivate this construction we start with a simple example.

For each limit $\alpha<\omega_1$, fix $c_{\alpha}\subseteq \alpha$ cofinal of order type $\omega$.
As before, we will view $c_{\alpha}$ both as a set and as an $\omega$-sequence which enumerates it
in increasing order. Thus, we will write $c_{\alpha}(n)$ for the $n$-th element of $c_{\alpha}$.
The sequence $\langle c_{\alpha}:  \alpha<\omega_1, \lim(\alpha)\rangle$ is called a $\vec c$-sequence.

We can generalize the definition of $\osc$ as follows: for $f,g\in (\omega_1)^{\leq \omega}$,
$$
\osc(f,g)=\vert \{n <\omega :  f(n)\leq g(n) \land f(n+1)>g(n+1)\}\vert.
$$
Given a subset $S$ of $\omega_1$ consisting of limit ordinals, let
$$
U_S=\{ s \in [\omega_1]^{<\omega}: \{\alpha\in S : s\sqsubseteq c_{\alpha}\} \textrm{ is stationary}\}.
$$

\begin{lemma}\label{pdl} Assume $S\subseteq \omega_1$ is stationary. Then $U_S$ is an $\omega_1$-superperfect tree.
\end{lemma}

\begin{proof}
Given $s\in U_S$ let $(U_S)_s = \{ t\in U_S: s\sqsubseteq t\}$ and let
$\alpha_{s,n}= \sup \{ t(n): t\in (U_S)_s\}$. Then there is $n$ such that
$\alpha_{s,n}=\omega_1$. To see this, assume otherwise and  let $\alpha=\sup \{ \alpha_{s,n}: n<\omega\}$.
Then $\alpha <\omega_1$. For each $\delta \in S\setminus (\alpha +1)$ such that $s\sqsubseteq c_{\delta}$
let $n_{\delta}$ be the least integer such that $c_{\delta}(n_\delta)>\alpha$.
By the Pressing Down Lemma, there is $t\in [\omega_1]^{<\omega}$ such that
$s\sqsubseteq t$ and
the set $\{ \delta \in S:   c_{\delta}\restriction (n_{\delta}+1)=t\}$
is stationary. It follows that $s\sqsubseteq t\in U_S$ and $\max(t) >\alpha$, a contradiction.
\qed
\end{proof}

\begin{lemma} Given two stationary sets $S,T\subseteq \omega_1$,
there is $n_0<\omega$ such that for all $k<\omega$ there exist $\alpha\in S$ and $\beta\in T$ such that
$\osc (c_{\alpha}, c_{\beta})=n_0+k$.
\end{lemma}

\begin{proof} By Lemma \ref{pdl} both $U_S$ and $U_T$ are $\omega_1$-superperfect.
Let $s$ and $t$ be the least $\omega_1$-splitting nodes of $U_S$ and $U_T$ respectively.
We may assume that $|s|\leq |t|$ and $s(|s|-1)\leq t(|s|-1)$.
Let $n_0=\osc(s,t)+1$. Now, as in the proof of Lemma \ref{estendereinodi},
given an integer $k$ we can find $\omega_1$-splitting
nodes $s'$ and $t'$ of $U_S$ and $U_T$ respectively, such that $s\sqsubseteq s'$,
$t\sqsubseteq t'$ and $\osc(s',t')=n_0+k-1$. Moreover, we can arrange that
$|s'|\leq |t'|$ and $s'(|s'|-1)\leq t'(|s'|-1)$. Now, pick any $\beta \in T$
such that $t'\sqsubseteq c_\beta$. Since $s'$ is an $\omega_1$-splitting node
of $U_S$, there is $\gamma >\beta$ such that $s'\, \wh \, \gamma \in U_S$.
Pick $\alpha \in S$ such that $s'\, \wh \, \gamma \sqsubseteq c_\alpha$.
It follows that $\osc (c_\alpha,c_\beta)=\osc(s',t')+1=n_0+k$, as desired.
\qed
\end{proof}

We can then improve $\osc$ as before to get some better coloring.
We know that our coloring cannot be as strong as in the case of $\mathfrak{b}$, since $\mathrm{MA}_{\aleph_1}$ implies
that the countable chain condition is productive, so we have to give up some of
the properties of our coloring.

We now present a construction of Moore \cite{Moore} of a coloring
of pairs of countable ordinals witnessing $\omega_1 \not\to [\omega_1 ; \omega_1]^2_{\omega}$
and use it to construct an $L$-space.
As before we fix  a sequence $\langle C_{\alpha} : \alpha<\omega_1\rangle$ such that
\begin{itemize}
\item if $\alpha= \xi+1$, then $C_{\alpha}= \{\xi\}$;
\item if $\alpha$ is limit, then $C_{\alpha}\subseteq \alpha$ is cofinal and of order type $\omega$.
\end{itemize}
Given $\alpha<\beta$ we define the \emph{walk} from $\beta$ to $\alpha$.
We first define a sequence $\beta_0>\beta_1\cdots>\beta_l=\alpha$ as follows:
\begin{itemize}
\item $\beta_0=\beta$;
\item $\beta_{i+1}=\min (C_{\beta_i}\setminus \alpha)$.
\end{itemize}
Then we define $\xi_0\leq \xi_1 \cdots\leq \xi_{l-1}$ by setting
$$
\xi_k=\max \overset{k}{\underset{j=0}{\bigcup}} (C_{\beta_j}\cap \alpha),
$$
for all $k\leq l-1$.
We call $\mathrm{Tr}(\alpha,\beta)=\{\beta_0,...,\beta_l\}$ the \emph{upper trace} and
$L(\alpha,\beta)=\{\xi_0,...,\xi_{l-1}\}$ the \emph{lower trace} of the walk from $\beta$
to $\alpha$.

$$\setlength{\unitlength}{1cm}
\begin{picture}(7.5,6.5)
\linethickness{0.4mm}
\put(0.4,2.7 ){\line(1,0){5.3}}
\put(5.3,5.5 ){\line(1,0){0.4}}
\thinlines
\put(5.8,2.6){$\alpha$}
\put(5.8,5.5){$\beta=\beta_0$}
\put(5.5,1){\line(0,1){4.5}}
\linethickness{0.4mm}
\put(5.3,3.5 ){\line(1,0){0.4}} 
\put(5.8,3.5){$\beta_1$}
\put(5.4,1.8 ){\line(1,0){0.2}}
\put(5.8,1.6){$\xi_0$}
\thinlines
\multiput(3.5,3.5)(0.4,0){5}{\line(1,0){0.2}}
\multiput(1.8,1.8)(0.4,0){10}{\line(1,0){0.2}}
\put(1.2, 1.8){$\xi_0$}
\put(2.75,3.5){$\beta_1$}
\put(1.5,1.5){\circle*{0.1}}
\put(1.8,1.8){\circle*{0.1}}
\put(3.5,3.5){\circle*{0.1}}
\put(4,4){\circle*{0.1}}
\put(4.6,4.6){\circle*{0.1}}
\put(4.9,4.9){\circle*{0.1}}
\put(5.1,5.1){\circle*{0.1}}
\qbezier(6,4)(6.5,4.6)(6,5.2)
\put(6,4){\line(0,1){0.2}}
\put(6,4){\line(1,0){0.2}}
\put(1.3,1.5){\line(1,1){4}}
\put(1.5,1.3){\line(1,1){4.2}}
\qbezier(1.3,1.5)(1.3,1.3)(1.5,1.3)
\put(4.3,5.3){$C_{\beta}$}
\put(3.5,1.5){\circle*{0.1}}
\put(4,2){\circle*{0.1}}
\put(4.8,2.8){\circle*{0.1}}
\put(5.2,3.2){\circle*{0.1}}
\put(5.4,3.4){\circle*{0.1}}
\put(3.3,1.5){\line(1,1){2}}
\put(3.5,1.3){\line(1,1){2.2}}
\qbezier(3.3,1.5)(3.3,1.3)(3.5,1.3)
%
\multiput(4,2)(0.4,0){4}{\line(1,0){0.2}}
\put(5.8,2.1){$\xi_1$}
\linethickness{0.4mm}
\put(5.4, 2){\line(1,0){0.2}}
\end{picture}
$$

\begin{lemma}\label{fact1} Suppose that $\alpha\leq \beta\leq \gamma$ and $\max (L(\beta,\gamma))<\min (L(\alpha,\beta))$,
then $L(\alpha, \gamma)=L(\alpha,\beta)\cup L(\beta,\gamma).$
\end{lemma}

\begin{proof} Since $\max (L(\beta,\gamma))<\min (L(\alpha,\beta))$, we have
$C_{\xi}\cap \alpha=C_{\xi}\cap \beta$ whenever $\xi$ is in $\mathrm{Tr}(\beta, \gamma)$
and $\xi \neq \beta$.
It follows that $\beta \in \mathrm{Tr}(\alpha,\gamma)$ and
$\mathrm{Tr}(\alpha, \gamma)=\mathrm{Tr}(\alpha,\beta)\cup \mathrm{Tr}(\beta,\gamma)$.
Assume $\mathrm{Tr}(\alpha,\gamma)=\{\gamma_0,...,\gamma_l\}$ and
$L(\alpha,\gamma)=\{\xi_0,...,\xi_{l-1}\},$ there is $l_0\leq l$ such that $\gamma_{l_0}=\beta$.
Therefore, $\{\xi_k\}_{k\leq l_0-1}=L(\beta,\gamma).$ On the other hand
$\max (C_{\gamma_{l_0}}\cap \alpha)>\xi_{l_0-1}$ because $\xi_{l_0-1}\in L(\beta,\gamma)$
and $\max C_{\gamma_{l_0}}\in L(\alpha,\beta)$, hence if $k\geq l_0,$ then
$$
\xi_k=\max \overset{k}{\underset {j=0}{\bigcup}} (C_{\gamma_j}\cap \alpha)=
\max \overset{k}{\underset {j=l_0}{\bigcup}} (C_{\gamma_j}\cap \alpha),
$$
and so $L(\alpha,\beta)=\{\xi_k\}_{k=l_0}^{l-1}$.
\qed
\end{proof}

\begin{lemma}\label{limit} If $\alpha<\beta,$ then $L(\alpha,\beta)$ is a non empty finite set and,
for every limit ordinal $\beta$, $\underset{\alpha\to\beta}{\lim} \min(L(\alpha,\beta))=\beta$.
\end{lemma}

\begin{proof} The first statement is trivial, let us prove that
$\underset{\alpha\to\beta}{\lim} \min(L(\alpha,\beta))=\beta$, for every limit ordinal $\beta$.
Given $\alpha<\beta$, one can take $\alpha'\in C_{\beta}\setminus (\alpha+1)$.
Then $\alpha<\alpha'=\max (C_{\beta}\cap (\alpha'+1))=
\min L(\alpha'+1,\beta)\leq \underset{\alpha\to\beta}{\lim} \min(L(\alpha,\beta))$.
It follows that $\beta\leq \underset{\alpha\to\beta}{\lim} \min(L(\alpha,\beta))\leq \beta$,
 and this completes the proof.
\qed
\end{proof}

Fix a sequence $\langle e_{\alpha}: \alpha<\omega_1\rangle$ satisfying the following conditions:

\begin{enumerate}
\item $e_{\alpha}: \alpha\to\omega$ is finite-to-one;\\
\item $\alpha<\beta$ implies $e_{\beta}\restr \alpha=_*e_{\alpha},$ i.e.
$\{\xi<\alpha:  e_{\beta}(\xi)\neq e_{\alpha}(\xi)\}$ is finite.
\end{enumerate}

Given $\alpha <\beta <\omega_1$ let $\sdiff(\alpha,\beta)$ be the least $\xi <\alpha$ such that 
$e_{\alpha}(\xi)\neq e_{\beta}(\xi)$, if it exists, and $\alpha$ otherwise. 
We define $\osc(\alpha,\beta)$ as follows
$$
\osc(\alpha,\beta)=\vert \{ i\leq l-1:
e_{\alpha}(\xi_i)\leq e_{\beta}(\xi_i) \land e_{\alpha}(\xi_{i+1})>e_{\beta}(\xi_{i+1})\}\vert
$$
where $L(\alpha,\beta)=\{\xi_0<\cdots<\xi_{l-1}\}$.

It will be convenient also to use the notation
$\mathrm{Osc}(e_{\alpha},e_{\beta}, L(\alpha,\beta))$ for the set
$\{\xi_i\in L(\alpha,\beta):  e_{\alpha}(\xi_i)\leq e_{\beta}(\xi_i) \land e_{\alpha}(\xi_{i+1})>e_{\beta}(\xi_{i+1}) \}$.

Our aim is to prove the following theorem due to J. Moore (see \cite{Moore}).

\begin{theorem}[\cite{Moore}]\label{moore} Let $A,B\subseteq \omega_1$ be uncountable,
then for all $n<\omega$ there exist $\alpha \in A$, $\beta_0,\beta_1,...,\beta_{n-1}\in B$ and $k_0$
such that $\alpha<\beta_0,...,\beta_{n-1}$ and $\osc(\alpha,\beta_m)=k_0+m,$ for all $m<n$.
\end{theorem}

This means that we can get arbitrary long intervals of oscillations with a fixed lower
point $\alpha\in A$. We can generalize this to get even more:

\begin{theorem}[\cite{Moore}]\label{justin} Given $\mathfrak{A}\subseteq [\omega_1]^k$ and
$\mathfrak{B}\subseteq [\omega_1]^l$ uncountable and parwise disjoint, and given $n<\omega$,
we can find $A\in \mathfrak{A}$ and $B_0,...,B_{n-1}\in \mathfrak{B}$ such that
$\max A<\min B_i$, for all $i<n$, and $\osc(A(i), B_m(j))=\osc(A(i), B_0(j))+m$ for all $i<k$,
 $j<l$ and  $m<n$.
\end{theorem}

In order to prove Theorem \ref{moore} we demonstrate the following lemma.


\begin{lemma}\label{Moore} Let $A,B\subseteq \omega_1$ be uncountable.
There exists a club $C\subseteq \omega_1$ such that if $\delta\in C$,
$\alpha\in A\setminus \delta$, $\beta\in B\setminus \delta,$ and $R\in \{=,>\}$,
then there are $\alpha'\in A\setminus \delta$ and $\beta'\in B\setminus \delta$
satisfying the following properties:
\begin{enumerate}
\item $\max L(\alpha,\beta)<\sdiff(\alpha,\alpha'), \sdiff(\beta,\beta')$;
\item $L(\delta,\beta)\sqsubseteq L(\delta,\beta')$;
\item for all $\xi\in L^+=L(\delta,\beta')\setminus L(\delta,\beta)$,
we have $e_{\alpha'}(\xi)\spazio R\spazio e_{\beta'}(\xi).$
\end{enumerate}
\end{lemma}

\begin{proof} Fix a sufficiently large regular cardinal $\theta$. We will
show that if $M \prec H_\theta$ is a countable elementary substructure
containing all the relevant objects, then $\delta =M\cap \omega_1$ satisfies
the conclusion of the lemma. Since the set of such $\delta$ contains
a club in $\omega_1$ this will be sufficient.
Thus, fix $M$ and $\delta$ as above and let $\alpha$ and $\beta$ be as
in the hypothesis of the lemma. We first suppose that $R$ is $=$.
Since $\delta$ is a limit ordinal, we can take $\gamma_0<\delta$ such that:

\begin{enumerate}
\item $\max (L(\delta,\beta))<\gamma_0;$
\item for all $\xi \in (\gamma_0,\delta),$ $e_{\alpha}(\xi)=e_{\beta}(\xi).$
\end{enumerate}

By Lemma \ref{limit} we can fix also $\gamma<\delta $ such that $\gamma_0<\min L(\xi,\delta)$,
for all $\xi\in (\gamma,\delta)$.
Let $D$ be the set of all $\delta'<\omega_1$ such that for some $\alpha'\in A\setminus \delta'$
and $\beta'\in B\setminus \delta'$ the following properties are satisfied:

\begin{itemize}
\item[$(a)$] $e_{\alpha'}\restr \gamma_0=e_{\alpha}\restr \gamma_0,$ $ e_{\beta'}\restr \gamma_0=e_{\beta}\restr \gamma_0;$
\item[$(b)$] $L(\delta', \beta')=L(\delta,\beta);$
\item[$(c)$] for all $\xi\in (\gamma,\delta'),$ $\gamma_0<\min L(\xi,\delta');$
\item[$(d)$] for all $\xi\in (\gamma_0,\delta'),$ $e_{\alpha'}(\xi)=e_{\beta'}(\xi).$
\end{itemize}

Observe that for all $\xi\geq \gamma_0,$ $e_{\xi}\restr \gamma_0$ is in M since, by definition,
$e_{\xi}\restr \gamma_0=_* e_{\gamma_0}.$ This means that $D$ is definable in $M,$ hence $D\in M$.
Moreover $D\not\subseteq M$ (since $\delta\in D$) hence $D$ is uncountable.
Choose $\delta'>\delta$ in $D$ with $\alpha'\in A\setminus \delta'$ and $\beta'\in B\setminus \delta'$
witnessing $\delta'\in D.$ By condition $(a)$ of the definition of $D,$

$$
\gamma_0\leq \sdiff(\alpha,\alpha'),\sdiff(\beta,\beta').
$$

Put $L^+=L(\delta,\delta')$, then $\max L(\delta,\beta)=\max L(\delta',\beta')<\min L^+,$ hence
$$L(\delta,\beta')=L(\delta', \beta')\cup L^+=L(\delta,\beta)\cup L^+.$$ Given $\xi\in L^+$,
by condition $(c)$, we have $\gamma_0<\min L^+\leq \xi.$ It follows that $\xi\in (\gamma_0,\delta')$,
so $(d)$ implies $e_{\alpha'}(\xi)=e_{\beta'}(\xi).$

Now assume that $R$ is $>.$ Let $E$ be the set of all limits
$\nu<\omega_1$ such that for all $\alpha_0\in A\setminus \nu$, $\nu_0<\nu,$ $\varepsilon< \omega_1$,
$n<\omega$ and finite $L^+\subseteq \omega_1\setminus \nu$,
there exists $\alpha_1\in A\setminus \varepsilon$ with $\nu_0\leq \sdiff(\alpha_0,\alpha_1)$ and
$e_{\alpha_1}(\xi)>n,$ for all $\xi \in L^+$. Since $E$ is definable from parameters in $M$
it follows $E\in M$, as well.

\begin{cl} The ordinal $\delta$ is in $E.$ In particular $E$ is uncountable.
\end{cl}

\begin{proof}
Let $\alpha_0,\nu_0,\varepsilon,n,L^+$ be given as in the definition of $E$ for $\nu =\delta$.
Since $e_{\alpha_0}$ is finite-to-one, we can assume w.l.o.g. that
$\nu_0>\sup \{\xi<\delta:  e_{\alpha_0}(\xi)\leq n\}.$ By the elementarity of $M$,
there exists $\delta'>\varepsilon,$ $\delta,$ $\max L^+$ and $\alpha_1\in A\setminus \delta'$
such that the following conditions hold:

\begin{itemize}
\item $e_{\alpha_0}\restr \nu_0=e_{\alpha_1}\restr \nu_0;$
\item for all $\xi$ in $(\nu_0,\delta'),$ we have $e_{\alpha_1}(\xi)>n.$
\end{itemize}

Since $L^+\subseteq \delta'\setminus \delta,$ this completes the proof of the claim.
\qed
\end{proof}

Now apply the elementarity of $M$ and the fact that $E$ is uncountable to find $\gamma_0\in E$
such that $L(\delta,\beta)<\gamma_0<\delta.$ By Lemma \ref{limit} we can find a $\gamma<\delta$ such that
if $\xi \in (\gamma,\delta),$ then $\gamma_0<L(\xi,\delta)$.
Again by the elementarity of $M$ select limit $\delta'>\delta$ and $\beta'\in B\setminus \delta'$
such that the following conditions hold:
\begin{itemize}
\item $e_{\beta'}\restr \gamma_0=e_{\beta}\restr \gamma_0$;
\item $L(\delta',\beta')=L(\delta,\beta)$;
\item $\gamma < \xi< \delta'$ implies $\gamma_0<L(\xi,\delta')$.
\end{itemize}

Put $L^+=L(\delta,\delta'),$ then $L^+\subseteq \omega_1\setminus \gamma_0.$ Since $\gamma_0\in E$
we can apply the definition of $E$ with $\nu_0=\max L(\delta,\beta)+1,$
$n=\max \{e_{\beta'}(\xi); \xi\in L^+\}$ to find $\alpha'\in A\setminus \delta$
such that for all $\xi\in L^+,$ $\max L(\delta,\beta)<\sdiff(\alpha,\alpha')$ and
$e_{\alpha'}(\xi)>e_{\beta'}(\xi).$ This completes the proof of Lemma \ref{Moore}.
\qed
\end{proof}

We can finally prove Theorem \ref{moore}.


\begin{proof}(of {\bf Theorem \ref{moore}}).
Let $A,B\subseteq \omega_1$ be two uncountable sets and let $M\prec H_{\aleph_2}$ be a
countable substructure containing everything relevant with $\delta=M\cap \omega_1$.
Since $M$ contains $A$ and $B,$ the club $C$ provided by Lemma \ref{Moore} is in $M$.
Use Lemma \ref{Moore} to select $\alpha_0,\alpha_1,\ldots,\alpha_n,\ldots$ in
$A\setminus \delta,$ $\beta_0,\beta_1,\ldots,\beta_n,\ldots$
in $B\setminus \delta$ and $\xi_0,\xi_1,...,\xi_n,...$ in $\delta$
such that for all $n<\omega$ the following conditions are satisfied:

\begin{enumerate}
\item $L(\delta,\beta_n)\sqsubset L(\delta, \beta_{n+1})$;
\item $\xi_n\in L(\delta, \beta_{n+1})\setminus L(\delta,\beta_n)$;
\item $\mathrm{Osc}(e_{\alpha_{n+1}}, e_{\beta_{n+1}}, L(\delta,\beta_{n+1}))=
\mathrm{Osc}(e_{\alpha_{n}}, e_{\beta_{n}}, L(\delta,\beta_{n}))\cup \{\xi_n\}$;
\item if $m>n,$ then $\xi_n<\sdiff(\alpha_m,\alpha_{m+1}),\sdiff(\beta_m,\beta_{m+1})$;
\item $e_{\alpha_n}(\max L(\delta, \beta_n))>e_{\beta_n}(\max L(\delta, \beta_n))$.
\end{enumerate}

Suppose $\alpha_n$ and $\beta_n$ have been defined.
We obtain $\alpha_{n+1}$ and $\beta_{n+1}$ by applying Lemma \ref{Moore} twice:
first with $R$ being $=$, second with $R$ being $>$.
If $\alpha'$ and $\beta'$ are the two ordinals obtained by applying the lemma the first time,
then $\xi_n=\min (L(\delta,\beta_{n+1})\setminus L(\delta,\beta'))$.

Now let $n$ be given, pick $\gamma_0<\delta$ such that

$$
\gamma_0>\max L(\delta,\beta_n),\max \{\xi<\delta:  \exists m,m'\leq n (e_{\beta_m}(\xi)\neq e_{\beta_{m'}}(\xi)) \}.
$$
Using the elementarity of $M$ and Lemma \ref{limit}, select $\alpha\in A\cap \delta$ such that
$$
\max L(\delta, \beta_n)<\sdiff(\alpha,\alpha_n)\textrm{ and }\gamma_0<\min L(\alpha,\delta).
$$
Now let $m<n$ be fixed. It follows from Lemma \ref{fact1} that
$$
L(\alpha,\beta_m)=L(\alpha,\delta)\cup L(\delta,\beta_m).
$$
Finally $e_{\beta_m}\restr L(\alpha,\delta)$ does not depend on $m$ since
$$
\min L(\alpha,\delta)>\gamma_0> \max \{\xi<\delta:  \exists m,m'\leq n (e_{\beta_m}(\xi)\neq e_{\beta_{m'}}(\xi)) \}.
$$
Therefore
$$
\mathrm{Osc}(e_{\alpha},e_{\beta_0}, L(\alpha,\delta))=\mathrm{Osc}(e_{\alpha},e_{\beta_m}, L(\alpha,\delta)).
$$
By $5.$, $\mathrm{Osc}(e_{\alpha},e_{\beta_m},L(\alpha,\beta_m))=
\mathrm{Osc}(e_{\alpha},e_{\beta_m}, L(\alpha,\delta))\cup \mathrm{Osc}(e_{\alpha}, e_{\beta_m}, L(\delta, \beta_m))$
so by  $3.,$ $\mathrm{Osc}(e_{\alpha},e_{\beta_m}, L(\alpha, \beta_m))=
\mathrm{Osc}(e_{\alpha}, e_{\beta_0}, L(\alpha, \beta_0))\cup \{\xi_{m'}; m'<m \}$.
Hence  $\osc(\alpha,\beta_m)=\osc(\alpha,\beta_0)+m$ and this completes the proof.
\qed
\end{proof}



By using the previous results we can, finally, prove the existence of an $L$-space,
that is ,a regular Hausdorff space which is hereditarily Lindel\"of but not hereditarily separable.
We will work in $\mathbb{T}=\{z\in \mathbb{C}:  \vert z\vert=1\}$.
We fix a sequence $\langle z_{\alpha}:  \alpha<\omega_1\rangle$ of rationally independent
elements of $\mathbb{T}$. It is easy to find such a sequence since
given any countable rationally independent subset $I$ of $\mathbb{T}$,
there are only countable many $z$ for which $I\cup \{z\}$ is rationally dependent.
Consider, now, the function defined as follows:

$$
o(\alpha,\beta)=z_{\alpha}^{\osc(\alpha,\beta)+1},
$$
for all $\alpha<\beta<\omega_1$.

We will use the \emph{Kronecker's Theorem} (see \cite{Kronecker} or \cite{Tchebycheff}) which is the following statement:

\begin{theorem} Suppose that $\langle z_i\rangle_{i<k}$ is a sequence of elements of $\mathbb{T}$
which are rationally independent. For every $\epsilon>0$, there is $n_{\epsilon}\in \mathbb{N}$
such that if $u,v\in \mathbb{T}^k,$ there is $m<n_{\epsilon}$ such that
$$
\vert u_i z_i^m-v_i\vert<\epsilon,
$$
for all $i<k$.
\end{theorem}

We can, now, define the $L$-space. For every $\beta<\omega_1,$ we define a function $w_{\beta}:\omega_1\to \mathbb{T}$ as follows:
$$
w_{\beta}(\xi)=\left\{ \begin{array}{ll}
o(\xi,\beta)&\textrm{ if $\xi<\beta$}\\
1&\textrm{ otherwise.}
\end{array}\right.
$$
Let $\mathscr{L}=\{w_{\beta}:  \beta<\omega_1\}$ viewed as a subspace of $\mathbb{T}^{\omega_1}$.

\begin{remark} $\mathscr{L}$ is not separable.
\end{remark}

For all $X\subseteq \omega_1,$ let $\mathscr{L}_{X}=\{w_{\beta}\restr X; \beta\in X\}$
viewed as a subspace of $\mathbb{T}^X.$ We will simply write $w_{\beta}$
for $w_{\beta}\restr X$ when referring to elements of $\mathscr{L}_{X}$.
Our aim is to prove that $\mathscr{L}_{X}$ is an $L$-space, for all $X$ uncountable.

\begin{lemma}\label{intorni} Let $\mathscr{A}\subseteq [\omega_1]^k$ and $\mathscr{B}\subseteq [\omega_1]^l$
be uncountable families of pairwise disjoint sets. For every sequence $\langle U_i\rangle_{i<k}$ of open neighborhoods
in $\mathbb{T}$ and every $\phi: k\to l$, there are $a\in \mathscr{A}$ and $b\in \mathscr{B}$ such that $\max (a)<\min (b)$ 
and for all $i<k$,
$$
o(a(i), b(\phi(i)))\in U_i.
$$
\end{lemma}

\begin{proof} We may assume without loss of generality that every $U_i$ is an $\epsilon$-ball about a point $v_i$,
for some fixed $\epsilon >0$. We can assume also that the integer $n_{\epsilon}$ of the Kronecker's Theorem
for the sequence $\langle z_{a(i)}\rangle_{i<k}$ is uniform for $a\in \mathscr{A}$.
Apply Theorem \ref{justin} to find $a\in \mathscr{A}$ and $\langle b_m\rangle_{m<n_{\epsilon}}$ a sequence of elements
of $\mathscr{B}$ such that
$$
a<b_m
$$
$$\osc(a(i), b_m(j))=\osc(a(i), b_0(j))+m,
$$
for all $i<k,j<l$ and $m<n_{\epsilon}.$
For each $i<k,$ put $u_i=o(a(i), b_0(\phi(i))).$ There is an $m<n_{\epsilon}$, such that
$$
\vert u_i z^m_{a(i)} -v_i\vert<\epsilon,
$$
for all $i<k$ or, equivalently, $o(a(i), b_m(\phi(i)))\in U_i$. This completes the proof.
\qed
\end{proof}

\begin{lemma}\label{injezioni} If $X,Y\subseteq \omega_1$ have countable intersection,
then there is no continuous injection from any uncountable subspace of
$\mathscr{L}_{X}$ into $\mathscr{L}_{Y}.$
\end{lemma}

\begin{proof} Suppose, by way of contradiction, that such an injection $g$ does exist.
Then there are an uncountable set $X_0\subseteq X$ and an injection
$f: X_0\to Y$ such that $g(w_{\beta})=w_{f(\beta)}$.
We may assume without loss of generality that $X_0$ is disjoint from $Y$.
For each $\xi<\omega_1$, let $\beta_{\xi}\in X_0$ and $\zeta_{\xi}\in Y$ be such that
$f(\beta_{\xi})>\zeta_{\xi}$ and if $\xi<\xi',$ then $\beta_{\xi}<\zeta_{\xi'}$.
Let $\Xi\subseteq \omega_1$ be uncountable such that for every $\xi\in \Xi$
there is an open neighborhood $V$ in $\mathbb{T},$ such that $g(w_{\beta_{\xi}})(\zeta_{\xi})\notin \bar{V}$.
Let $W_{\xi}=\{w\in \mathscr{L}_{Y}:  w(\zeta_{\xi})\notin \bar{V}\},$ for all $\xi<\omega_1$.
Since $g$ is continuous at $w_{\beta_{\xi}}$,  there is a basic open neighborhood $U_{\xi}$
of $w_{\beta_{\xi}}$ such that $U_{\xi}\subseteq g^{-1}W_{\xi}.$ By applying the $\Delta$-system lemma
and the second countability of $\mathbb{T},$ we can find an uncountable $\Xi'\subseteq \Xi$,
a sequence of open neighborhoods $\langle U_i\rangle_{i<k}$ in $\mathbb{T},$ and $a_{\xi}\in [X]^k$
such that for all $\xi\in \Xi',$ the following conditions hold:
\begin{itemize}
\item $\{a_{\xi}\}_{\xi\in \Xi'}$ is a $\Delta$-system with root $a$;
\item $w_{\beta_{\xi}}\in \{w\in \mathscr{L}_{X}:  \forall i<k(w(a_{\xi}(i))\in U_i)\}\subseteq U_{\xi}$;
\item the inequality $\beta_{\xi}<f(\beta_{\xi})$ does not depend on $\xi$;
\item $\vert \zeta_{\xi}\cap a_{\xi}\vert$ does not depend on $\xi$.
\end{itemize}

Let $\mathscr{A}= \{a_{\xi}\cup \{\xi\}\setminus a\}_{\xi\in \Xi'}$ and
$\mathscr{B}=\{\beta_{\xi}, f(\beta_{\xi})\}_{\xi\in \Xi'}$.
By applying Lemma \ref{intorni} we can find $\xi<\xi'$ in $\Xi'$ such that for all $i<k,$
$$a_{\xi}\cup \{\zeta_{\xi}\}<\min (\beta_{\xi'}, f(\beta_{\xi'})),$$
$$w_{\beta_{\xi'}}(a_{\xi}(i))=o(a_{\xi}(i), \beta_{\xi'})\in U_i,$$
$$g(w_{\beta_{\xi'}})=w_{f(\beta_{\xi'})}(\zeta_{\xi})=o(\zeta_{\xi}, f(\beta_{\xi'}))\in V.$$
We have that $w_{\beta_{\xi'}}\in U_{\xi}$ and $g(w_{\beta_{\xi'}})\notin W_{\xi}.$ Contradiction.
\qed
\end{proof}

\begin{theorem}[\cite{Moore}] For every $X,$ $\mathscr{L}_{X}$ is hereditarily Lindel\"of.
\end{theorem}

\begin{proof} If not, then $\mathscr{L}_{X}$ would contain an uncountable discrete subspace.
Moreover it would be possible to find disjoint $Y,Z\subseteq X$ such that $\mathscr{L}_{Y}$
and $\mathscr{L}_{Z}$ contain uncountable discrete subspaces.
It is well known that any function from a discrete space to another discrete space is continuous
and this contradicts Lemma \ref{injezioni}.
\qed
\end{proof}

\begin{corollary}[\cite{Moore}] There exists an $L$-space, i.e. a hereditary Lindel\"of non separable $T_3$
topological space.
\qed
\end{corollary}
\bibliographystyle{abbrv}
\bibliography{osc}

\end{document}